\documentclass[11pt,a4paper]{article}

\usepackage{amsmath}
\usepackage{amsthm}
\usepackage{amsfonts}
\usepackage{amssymb}
\usepackage{amscd}
\usepackage{stmaryrd}
\usepackage[all]{xy}
\SelectTips{cm}{10}
\usepackage{graphics}
\usepackage{graphicx}
\usepackage{fancyhdr}
\usepackage{ushort}
\usepackage{url}

\usepackage{ifthen}
\usepackage[strict]{changepage}
\usepackage{mparhack}

\newcommand{\bbZ}{\mathbb Z}

\newcommand{\mcA}{\mathcal A}

\newcommand{\mcC}{\mathcal C}
\newcommand{\mcD}{\mathcal D}

\newcommand{\mcP}{\mathcal P}

\newcommand{\mcR}{\mathcal R}
\newcommand{\mcS}{\mathcal S}

\newcommand{\mcV}{\mathcal V}

\newcommand{\Ab}{\mathsf{Ab}}

\newcommand{\Cat}{\mathsf{Cat}}

\newcommand{\Ch}{\mathsf{Ch}}

\newcommand{\ob}{\mathop\mathsf{ob}\nolimits}

\newcommand{\VCat}{\mcV\textrm{-}\Cat}

\newcommand{\op}{\mathrm{op}}

\newcommand{\LRa}{\Leftrightarrow}

\newcommand{\xra}[1]{\xrightarrow{#1}}

\newcommand{\xlra}[1]{\xrightarrow{\ #1\ }}
\newcommand{\xlla}[1]{\xleftarrow{\ #1\ }}
\newcommand{\lra}{\longrightarrow}

\newcommand{\dra}{\xymatrix@1{\ar@{-->}[r]&}}

\newdir{c}{{}*!/-5pt/@^{(}}
\newdir{d}{{}*!/-5pt/@_{(}}
\newdir{ >}{{}*!/-5pt/@{>}}
\newdir{s}{{}*!/+10pt/@{}}
\newdir{|>}{{}*!/2pt/@{|}*@{>}}

\newcommand{\cof}[1][]{\mathbin{\:\!\!\xymatrix@1@C=15pt{{}\ar@{ >->}[r]^{#1} & {}}}}
\newcommand{\fib}[1][]{\mathbin{\:\!\!\xymatrix@1@C=15pt{{}\ar@{->>}[r]^{#1} & {}}}}
\newcommand{\embed}[1][]{\mathbin{\:\!\!\xymatrix@1@C=15pt{{}\ar@{c->}[r]^{#1} & {}}}}
\newcommand{\rla}[1][]{\mathbin{\:\!\!\xymatrix@1@C=15pt{{}\ar@<2pt>[r]^{#1} & {}\ar@<2pt>[l]}}}

\newcommand{\pbscale}{.25}
\newcommand{\pboffset}{.5}

\newcommand{\xycorner}[3]{\save [];#2**{}?(\pbscale)="a",[];#1**{}?(\pbscale);"a"**{}?(\pboffset);"a"**\dir{-},[];#3**{}?(\pbscale);"a"**{}?(\pboffset);"a"**\dir{-} \restore}

\newcommand{\pb}[1][{1}]{\xycorner{[r]**{}?(#1)}{[dr]**{}?(#1)}{[d]**{}?(#1)}}

\newcommand{\ev}{\mathop\mathrm{ev}\nolimits}

\newcommand{\Hom}{\mathop\mathrm{Hom}\nolimits}

\newcommand{\leftbox}[2]{{}\phantom{#1} \save []+L*+<.5pc>!!<0pt,\the\fontdimen22\textfont2>!L{#1#2} \restore}
\newcommand{\rightbox}[2]{{}\phantom{#2} \save []+R*+<.5pc>!!<0pt,\the\fontdimen22\textfont2>!R{#1#2} \restore}

\entrymodifiers={+!!<0pt,\the\fontdimen22\textfont2>}
\theoremstyle{plain}
\newtheorem{theorem}{Theorem}

\newtheorem{proposition}[theorem]{Proposition}

\theoremstyle{definition}
\newtheorem{definition}[theorem]{Definition}

\newtheorem{example}[theorem]{Example}

\theoremstyle{remark}
\newtheorem*{remark}{Remark}

\newtheoremstyle{tiny}
     {10pt}
     {10pt}
     {\footnotesize}
     {}
     {\itshape}
     {.}
     {.5em}
     {}

\theoremstyle{tiny}

\usepackage[utf8]{inputenc}

\newcommand{\MC}{\operatorname{\mathrm{MC}}}
\newcommand{\gA}{\mathsf{g}\mcA}
\newcommand{\dgA}{\mathsf{dg}\mcA}
\newcommand{\gZMod}{\mathsf{g}\Ab}
\newcommand{\dgZMod}{\mathsf{dg}\Ab}

\newcommand{\Sps}{\mcS_\mathsf{ps}}
\newcommand{\Spsloc}{\mcS_\mathsf{ps}^\mathsf{loc}}
\newcommand{\Spt}{\mcS_\mathsf{pt}}

\begin{document}

\author{L.\ Vok\v{r}\'{i}nek\footnote{%
The research of L.~V.~was supported by the Grant agency of the Czech republic under the grant 22-02964S.}}
\title{Enriched categorical aspects \\ of homological perturbation theory}
\date{\today}
\maketitle

\begin{abstract}
The paper presents an enriched categorical account of homological perturbation theory, including the formulation, proof and functoriality properties of the homological perturbation lemma.
\end{abstract}

\section{Introduction}

Homological perturbation theory originated in the study of homology of fibre bundles, see e.g.\ \cite{Shih}. Later, it found applications to homotopy coherent structures, since it allows to transfer such a structure along homotopy equivalences, see e.g.\ \cite{transfer}. In connection to these results, compatibility with various algebra and coalgebra structures was explored heavily. Similarly, in the original application to fibre bundles \cite{Shih}, compatibility with the maps in the Eilenberg-Zilber equivalence were of great importance. In algorithmic topology, see e.g.\ \cite{polypost}, perturbation lemma is used in order to construct recursively small models for chain complexes of various spaces, such as Eilenberg-MacLane spaces and their twisted cartesian products, hence Postnikov stages. Conceptually, and also with implementation in mind, it is useful to decompose the occuring homotopy equivalences into simpler pieces, e.g.\ the chain complex of a cartesian product admits a homotopy equivalence
\[C_*(B \times F) \simeq C_*B \otimes C_*F \simeq C_*^\mathrm{ef} B \otimes C_*^\mathrm{ef} F,\]
where the chain complexes decorated by $\mathrm{ef}$ are some small models for chain complexes of the base or the fibre. Now replacing the cartesian product by a twisted cartesian product amounts to perturbing the differential of the source. The perturbation lemma applied consecutively to the two occuring homotopy equivalences first propagates this perturbation to a perturbation on the tensor product $C_*B \otimes C_*F$ and then further to a perturbation on $C_*^\mathrm{ef} B \otimes C_*^\mathrm{ef} F$. What if we first compose the homotopy equivalences and then propagate the perturbation in a single step -- do we get the same perturbation? It would be really cumbersome if we had to deal with questions like this all the time. Our primary reason to study functoriality properties of the perturbation lemma was to settle down this and similar questions. Our secondary goal was to set up a categorical framework where perturbations have a reasonable definition and the perturbation lemma has a reasonable formulation (and ideally proof as well). A categorical definition of a perturbation reveals that it is an absolute limit. We conjecture that together with shifts (suspensions and desuspensions), biproducts and retracts, they generate all absolute limits. One of the consequences would be a concrete description of a Cauchy completion of a dg-category and thus of Morita equivalence.

In more detail, we define a $\delta$-perturbation $A_\delta$ of an object $A \in \mcC$, where $\delta \in \mcC(A, A)_{-1}$ satisfies the Maurer-Cartan equation $D\delta = -\delta^2$, by its universal mapping property. The object $A$ together with the perturbation $\delta$ forms a dg-diagram $\bbZ[\delta] \to \mcC$ with $\bbZ[\delta]$ a one object dg-category, i.e.\ a differential graded algebra, generated by $\delta$ of degree $-1$ subject to $d\delta = -\delta^2$. We describe weights $W$ and $W'$ such that the $\delta$-perturbation $A_\delta$ equals the $W$-weighted limit of this diagram and also the $W'$-weighted colimit of this diagram, thus an absolute limit.

Further, we define a double category associated with $\mcC$ whose underlying horizontal category is $\mcC$, whose vertical maps are the strong deformation retractions in $\mcC$ and whose squares are natural transformations between strong deformation retractions. This double category forms a natural framework for formulating basic building stones of the perturbation lemma as pushing or pulling vertical maps along certain horizontal maps. Concretely, for every strong deformation retraction $(f,g,h)$ and every perturbation $\delta$ of the source satisfying a natural nilpotency condition, there exists a map $\theta$ in the underlying $\gZMod$-enriched category (i.e.\ not respecting the differentials) such that the indicated filler
\[\xymatrix{
A \ar[r]^-\theta \ar[d]_-{(f,g,h)} & A_\delta \ar@{-->}[d]^-{(\widehat f, \widehat g, \widehat h)} \\
B \ar@{-->}[r] & B_{\delta'}
}\]
exists, in particular making up a new strong deformation retraction $(\widehat f, \widehat g, \widehat h)$ from $A_\delta$ to $B_{\delta'}$, i.e.\ a perturbation of the original strong deformation retraction $(f,g,h)$. We suppose that $\theta$ is not unique and this makes proving functoriality properties of this construction quite hard and tedious. In the last part, we prove compatibility with horizontal composition, vertical composition and tensor product.

\section{Notation}

Let $\mcA$ be an abelian category. We denote by $\gA$ the category of $\bbZ$-graded objects in $\mcA$, and by $\dgA$ the category of differential $\bbZ$-graded objects in $\mcA$, i.e.\ (unbounded) chain complexes in $\mcA$. In particular, we have the category of (unbounded) chain complexes of abelian groups, which we denote for simplicity by $\Ch = \dgZMod$. This is naturally a closed symmetric monoidal category that we will use as a basis for enrichment -- we will refer to $\Ch$-enriched categories as dg-categories; in particular, $\dgA$ is a dg-category.

We denote the differential of a chain complex generally as $d$. However, for $\Hom$-complexes we use $D$ to avoid confusion between $Df$, i.e.\ $D$ applied to $f$, and $df$, i.e.\ $d$ composed with $f$. The forgetful functor $\dgZMod \to \gZMod$ preserves tensor product and $\Hom$-complex strictly, i.e.\
\[(A \otimes B)_n = \bigoplus_{k+l=n} (A_k \otimes B_l),\quad \Hom(A, B)_n = \prod_{k+n=l} \Hom(A_k, B_l)\]
and these are equipped with differentials $d = d \otimes 1 + 1 \otimes d$ on the tensor product (using the Koszul sign convention\footnote{In general, all Koszul signs come from the symmetry $A \otimes B \xlra\cong B \otimes A$ given by $x \otimes y \mapsto (-1)^{|x| \cdot |y|} y \otimes x$.} $(f \otimes g)(x \otimes y) = (-1)^{|g| \cdot |x|} f x \otimes g y$) and $D = [d,-] = d_* - d^*$ on the $\Hom$-complex (i.e.\ the commutator, again using the Koszul sign convention $d_* f = d f$ and $d^* f = (-1)^{|d| \cdot |f|} f d$; in our case $|d| = -1$). In terms of the $\Hom$-complex, maps of the underlying ordinary category of a dg-category are elements of $Z_0 \Hom(A, B)$, i.e.\ $0$-cycles, while maps of the underlying ordinary category of a $\gZMod$-enriched category are elements of $\Hom(A,B)_0$, i.e.\ $0$-chains; the former will be called maps, or dg-maps for emphasis, while the latter will be called non-dg-maps. More generally, elements $f \in \Hom(A,B)_n$ will be called non-dg-maps of degree $n = |f|$. Accordingly, invertible elements will be called (dg-)isomorphisms and non-dg-isomorphisms.

The following compatibility relations hold in any $\gZMod$-enriched category for the same reason they hold in any graded associative algebra and its associated graded Lie algebra:
\begin{align*}
[\delta, f g] & = [\delta, f] g + (-1)^{|\delta|\cdot|f|} f [\delta, g] \\
[\delta, [f, g]] & = [[\delta, f], g] + (-1)^{|\delta|\cdot|f|} [f, [\delta, g]]
\end{align*}
whenever all necessary compositions are defined, e.g.\ when $\delta, f, g$ are non-dg-transformations from the identity. Specializing to $\delta = d$, we obtain the following compatibility relations between the differential of the $\Hom$-complex and the composition. These also hold in any dg-category, simply because the composition is required to be a chain map:
\begin{align*}
D(f g) & = D f \cdot g + (-1)^{|f|} f \cdot D g \\
D[f, g] & = [D f, g] + (-1)^{|f|} [f, D g]
\end{align*}
Finally, we list a number of properties that easily follow from the ones above, where we denote by $D_\ell \alpha = \alpha^{-1} \cdot D \alpha$ the left logarithmic derivative of $\alpha$, by $D_r \alpha = D \alpha \cdot \alpha^{-1}$ the right logarithmic derivative of $\alpha$,
and by $c_\alpha = \alpha_* \cdot (\alpha^{-1})^*$ the conjugation by $\alpha$; in all cases, $\alpha$ is required to be a non-dg-isomorphism, i.e.\ invertible, but not necessarily a dg-map:
\begin{align*}
D(\alpha^{-1}) & = -(-1)^{|\alpha|} \cdot \alpha^{-1} \cdot D\alpha \cdot \alpha^{-1} \\
D(D_\ell \alpha) & = -(-1)^{|\alpha|} \cdot (D_\ell \alpha)^2 \\
D(D_r \alpha) & = (D_r \alpha)^2 \\
D(f_*) & = (D f)_* \\
D(f^*) & = (D f)^* \\
D_\ell(c_\alpha) & = (-1)^{|\alpha|} \cdot [D_\ell \alpha, -] \\
D_r(c_\alpha) & = [D_r \alpha, -]
\end{align*}
The second property says that the left logarthmic derivative $\delta = D_\ell \alpha$ of any non-dg-isomorphism of degree $0$ satisfies the Maurer-Cartan equation $D\delta = -\delta^2$. We denote the set of all solutions as
\[\MC(A) \subseteq \mcC(A, A)_{-1}.\]

A simple corollary of these formulas is the following proposition.

\begin{proposition} \label{prop:transfer_diagram}
Let $F \colon \mcA \to \mcC$ be a dg-functor. Let $Ga$ be a collection of objects of $\mcC$ indexed by objects $a \in \mcA$ and $\alpha_a \colon Fa \to Ga$ an equally indexed collection of non-dg-isomorphisms (of degree 0). Then
\[G \colon \mcA(a', a) \xlra{F} \mcC(Fa', Fa) \xlra{c_\alpha} \mcD(Ga', Ga)\]
makes $G$ into a dg-functor if and only if $D_\ell \alpha$ commutes, in the graded sense, with all maps in the image of $F$. It is enough that the commutation holds for the images of some dg-generators of $\mcA$.
\end{proposition}

\begin{proof}
The association $G$ easily respects the composition and identity (no conditions used). It thus remains to show that the map $G$ in the statement is a dg-map. Since $F$ is a chain map by assumption, it is thus required that $c_\alpha$ is a chain map on the image of $F$ and this is exactly $D_\ell(c_\alpha) = [D_\ell \alpha, -] = 0$ on this image. The last point follows from the formulas above; the least trivial is the closure under the differential, i.e.\ given that $D_\ell \alpha$ commutes with $f$ in the image, it also commutes with $D f$:
\begin{align*}
[D_\ell \alpha, Df] & = -D\underbrace{[D_\ell \alpha, f]}_0 + [D D_\ell \alpha, f] = [-(D_\ell \alpha)^2, f] \\
& = -D_\ell \alpha \cdot \underbrace{[D_\ell \alpha, f]}_0 - (-1)^{|f|} \cdot \underbrace{[D_\ell \alpha, f]}_0 \cdot D_\ell \alpha = 0.\qedhere
\end{align*}
\end{proof}


\section{Perturbations}

The general definition of a perturbation is based on the situation in the dg-category $\Ch$, where each chain complex $A \in \Ch$ is equipped with a differential $d$ and it thus make sense to consider the same graded abelian group equipped with a differential $d+\delta$; we denote this chain complex by $A_\delta$. The identity map on the underlying graded abelian group is a non-dg-map $\varphi \colon A_\delta \to A$. Its universal property is expressed in terms of the postcomposition map $\varphi_* \colon \Ch(B, A_\delta) \to \Ch(B, A)$. The differential on the left hand side is
\[(d+\delta)_* - d^* = (d_* - d^*) + \delta_*,\]
which is equal to the differential on $\Ch(B, A)_{\delta_*}$; we thus obtain an isomorphism
\[\varphi_* \colon \Ch(B, A_\delta) \xlra\cong \Ch(B, A)_{\delta_*}.\]
In other words, $A_\delta$ is the representing object for $\Ch(-, A)_{\delta_*}$. The exact same definition makes sense in any dg-category, specifying $A_\delta$ only up to isomorphism (thus producing a notion that is not evil).

\begin{definition}
Let $\delta \in \mcC(A,A)_{-1}$. A non-dg-map $\varphi \colon A' \to A$ expresses $A'$ as a $\delta$-perturbation of $A$ if the induced map
\[\varphi_* \colon \mcC(B, A') \xlra\cong \mcC(B, A)_{\delta_*}.\]
is an isomorphism. In this case we denote $A'$ by $A_\delta$.
\end{definition}

Since $\varphi$ is a map of the underlying $\gZMod$-enriched category inducing isomorphism of the respective representable functors, it is invertible by Yoneda lemma, so $\varphi$ must be a non-dg-isomorphism.

\begin{proposition}
A non-dg-isomorphism $\varphi \colon A' \to A$ expresses $A'$ as a $\delta$-perturbation of $A$ if and only if $D_r \varphi = -\delta$.
\end{proposition}

\begin{proof}
By naturality of $\varphi_*$, it is enough to check the dg-condition for $\varphi_*$ at the identity $1 \in \mcC(A', A')$. Since $D1 = 0$, this is equivalent to $(D+\delta_*) \varphi = 0$ $\LRa$ $D \varphi = - \delta \varphi$ $\LRa$ $D_r \varphi = -\delta$.
\end{proof}

This means that the notion of a perturbation, if we allow $\delta$ to vary, is the same as the notion of a non-dg-isomorphism. The existence question, i.e.\ whether a $\delta$-perturbation exists for arbitrary $\delta$ will be treated shortly. Since $\delta = -D_r \varphi$ satisfies the Maurer-Cartan equation, we thus ask whether the natural map
\[-D_r \colon \mcC(-, A)_0^\times / \mathrm{iso} \lra \MC(A),\]
from non-dg-isomorphisms into $A$ up to dg-isomorphisms, is bijective (clearly, two non-dg-isomorphisms have the same right logarthmic derivative if and only if they are both $\delta$-perturbations for the same $\delta$ if and only if they are isomorphic).

Consider the inverse $\psi \colon A \to A_\delta$ of the non-dg-isomorphism $\varphi \colon A_\delta \to A$. Easily $D_\ell \psi = - D_r \varphi = \delta$ and the dual of the last proposition provides an equivalent condition that
\[\psi^* \colon \mcC(A_\delta, B) \xlra\cong \mcC(A, B)_{-\delta^*}\]
is an isomorphism, so that $\psi$ satisfies the dual definition of a $\delta$-perturbation. The notion of a perturbation is thus self-dual and we will use the same notation for the dual notion (instead of $A^\delta$ or something alike). In addition, $D_r \psi = c_\psi D_\ell \psi = c_\psi \delta$ and we may thus view $\psi$ as a perturbation $\varphi \colon (A_\delta)_{-c_\psi \delta} \to A_\delta$ in the previous sense. To avoid cumbersome notation, it is useful to pretend that $\varphi$ is identity, thus also $\psi$, making $A$ into a $(-\delta)$-perturbation of $A_\delta$. More generally, the composition
\[A_\delta \xlla\psi A \xlra\psi A_\varepsilon\]
is a non-dg-isomorphism with left logarithmic derivative $c_\psi(\varepsilon - \delta)$ and can thus be seen as an $(\varepsilon - \delta)$-perturbation: $A_\varepsilon = (A_\delta)_{\varepsilon - \delta}$.

We know that $\delta = D_\ell \psi$ satisfies the Maurer-Cartan equation $D \delta = - \delta^2$. Our main question now is whether any $\delta \in \MC(A)$ admits a perturbation, i.e.\ whether there exists a non-dg-isomorphism $\psi \colon A \to A'$ with $D_\ell \psi = \delta$ (a sort of integrability question).

An object of $\mcC$ equipped with a perturbation can be seen as a diagram in $\mcC$ parametrized by a one object dg-category, i.e.\ a dga -- it is generated by a single element $\delta$ of degree $-1$ subject to $d\delta = -\delta^2$; we denote it simply by $\bbZ[\delta]$ and keep the differential implicit.

\begin{proposition}
There exists a weight $W \in [\bbZ[\delta], \Ch]$ such that $A_\delta \cong \{W, A\}_\mcP$. There exists a weight $W' \in [\bbZ[\delta]^\op, \Ch]$ such that $A_\delta \cong W *_\mcP A$. Consequently, $\delta$-perturbation is an absolute limit.
\end{proposition}

\begin{proof}
We start with the situation in $\Ch$, so that objects with a perturbation are exactly left dg-$\bbZ[\delta]$-modules, and we will express the perturbed complex as a weighted limit, i.e.\ as $\Hom_{\bbZ[\delta]}(W, -)$ from some module $W$. We start with the ``Yoneda isomorphism'' given by the evaluation at $1 \in \bbZ[\delta]$,
\[\ev_1 \colon \Hom_{\bbZ[\delta]}(\bbZ[\delta], C) \to C\]
under which the inverse image of $\delta$ is $(\delta^*)^*$, the precomposition with the right multiplication by $\delta$, since
\[\ev_1 (\delta^*)^* \varphi = (\delta^*)^* \varphi 1 = (-1)^{|\varphi|} \cdot \varphi \delta^* 1 = (-1)^{|\varphi|} \cdot \varphi (\underbrace{1 \cdot \delta}_{\delta \cdot 1}) = \delta \varphi 1 = \delta \ev_1 \varphi.\]
Therefore, we get an induced isomorphism upon perturbation
\[\Hom_{\bbZ[\delta]}(\bbZ[\delta]_{-\delta^*}, C) \cong \Hom_{\bbZ[\delta]}(\bbZ[\delta], C)_{(\delta^*)^*} \xlra{\ev_1} C_\delta\]
We denote the perturbed module as $W = \bbZ[\delta]_{-\delta^*}$. We have $d1 = 0 - \delta^*1 = -\delta$ and the rest is determined by $W$ being a left dg-$\bbZ[\delta]$-module (which follows from $-\delta^*$ being a map of left modules).

Now we can transport this result easily to an arbitrary dg-category $\mcC$:
\[\mcC(B, A_\delta) \cong \mcC(B, A)_{\delta_*} \cong \Hom_{\bbZ[\delta]}(W, \mcC(B, A)) \cong \mcC(B, \{W, A\}_{\bbZ[\delta]}),\]
yielding that the $\delta$-perturbation $A_\delta$ is the $W$-weighted limit
\[A_\delta \cong \{W, A\}_{\bbZ[\delta]}.\]
Dually, $W' = \bbZ[\delta]_{\delta_*}$ is a right dg-$\bbZ[\delta]$-module (with $d1 = 0 + \delta_* 1 = \delta$) such that
\[A_\delta \cong W' *_{\bbZ[\delta]} A,\]
i.e.\ the $\delta$-perturbation $A_\delta$ is the $W'$-weighted colimit. Put together, $A_\delta$ is an absolute colimit.
\end{proof}



Since we will be interpreting a non-dg-map $\alpha \colon A \to B$ in many ways as a non-dg-map between various perturbations of $A$ and $B$, we denote these systematically as $\alpha^\delta_\varepsilon \colon A_\delta \to B_\varepsilon$. If $A$ or $B$ is left unperturbed, we leave out the corresponding index.

In particular, any non-dg-isomorphism $\alpha \colon A \to B$ can be decomposed as
\[\alpha \colon A \xlra{\alpha_\varepsilon} B_\varepsilon \xlra{\varphi} B,\]
where the first map is a dg-isomorphism for $\varepsilon = -D_r \alpha$. In this way, any such non-dg-isomorphism can be seen as a dg-isomorphism upon perturbing the codomain. Dually, it can be decomposed as
\[\alpha \colon A \xlra{\psi} A_\delta \xlra{\alpha^\delta} B,\]
where the second map is a dg-isomorphism for $\delta = D_\ell \alpha$.


\begin{remark}
The paper \cite{dg_cauchy} gives a characterization of absolute limits in dg-categories in terms of (de)suspensions, biproducts, mapping cones and ``cokernels of protosplit chain maps''. We believe that one should be able to replace this by (de)suspensions, biproducts, retracts and perturbations. In addition, the Cauchy completion of a dg-category should be obtainable by succesively closing the dg-category under these constructions.
\end{remark}

\section{Double categories of diagrams}

We will now introduce the double categorical setup, where the perturbation lemma naturally takes place. This kind of double categories was considered for different reasons in~\cite[Section 4.3]{AWFS1}.

We are interested in the following situation. Let $\mcA$ be a cocategory object in $\VCat$ such that $\mcA^0$ is the unit $\mcV$-category $I$, i.e.\ one consisting of one object $*$ and $\mcA^0(*,*) = I$.
Then for every $\mcC \in \VCat$ the functor $\mcV$-category $[\mcA, \mcC]$ is a category object in $\VCat$, i.e.\ a double category that is enriched over $\mcV$ in some directions. More precisely, we have 
\[[\mcA, \mcC]_0 = [\mcA^0, \mcC] \cong \mcC\]
so the underlying horizontal $\mcV$-category is just $\mcC$. Vertical morphisms from $c_1$ to $c_0$ form the set given by the pullback
\[\xymatrix{
[\mcA, \mcC](c_1, c_0) \ar[rr] \ar[d] \pb & & {*} \ar[d]^-{(c_1, c_0)} \\
\ob [\mcA^1, \mcC] \ar[r]^-{(d_1, d_0)} & \ob [\mcA^0, \mcC] \times \ob [\mcA^0, \mcC] \ar@{}[r]|-{\cong} & \ob \mcC \times \ob \mcC
}\]
i.e.\ the set of $\mcV$-functors $\mcA^1 \to \mcC$ that restrict to $c_1$ and $c_0$ via the maps $d_i = (d^i)^*$ (precomposition with $d^i \colon \mcA^0 \to \mcA^1$). The $\mcV$-object of squares from $f_1$ to $f_0$ is given by the $\mcV$-object of $\mcV$-natural transformations. The vertical source and target maps are given by restrictions to the respective objects $\mcA^0 \xra{d^i} \mcA^1$.

We organize the data in the following way, where we use $|||$ to suggest that the arrows in this direction form a $\mcV$-object. Of course, we mainly think of $\mcV = \Ch$, in which case the arrows in these directions are chain complexes.
\[\xymatrix@R=1ex{
c_1' \ar[dd]_-{f_1} \ar@/_1ex/[rr] \ar@/^1ex/[rr] \ar@{}[rr]|-{|||} & & c_0' \ar[dd]^-{f_0} \\
\phantom{|} \ar@/_1ex/@{=>}[rr] \ar@/^1ex/@{=>}[rr] \ar@{}[rr]|-{|||} & & \\
c_1 \ar@/_1ex/[rr] \ar@/^1ex/[rr] \ar@{}[rr]|-{|||} & \phantom{|} & c_0
}\]

\begin{example}
Here $\mcA$ is the free living arrow interpreted as a free $\mcV$-category. In this way, one obtains a double category of squares associated with every $\mcV$-category $\mcC$.
\end{example}

\begin{example}
See also \cite[Definition 19]{AWFS2}. Here $\mcS$ is the dg-category with two objects $s=d^1\mcS$ and $t=d^0\mcS$ and morphisms freely generated by $f \colon s \to t$, $g \colon t \to s$ of degree $0$ and $h \colon s \to s$ of degree $1$ subject to
\begin{align*}
Df & = 0, & Dg & = 0, & Dh & = 1 - gf, \\
0 & = 1 - fg, & fh & = 0, & hg & = 0, & hh & = 0.
\end{align*}
The cocomposition is a certain (obvious) functor $\mcS \to \mcS +_I \mcS$. In order to describe it, we write maps in the first copy of the coproduct as $f_2$ etc., maps in the second copy as $f_0$ etc.\ and maps in the image of the cocomposition as $f_1$ etc.; then the functor is given by formulas $f_1 = f_0f_2$, $g_1 = g_2g_0$ and $h_1 = h_2 + g_2h_0f_2$. This yields as $[\mcS, \mcC]$ the double category with objects and horizontal maps those of $\mcC$, with vertical maps strong deformation retracts in $\mcC$, and with squares dg-natural transformations of strong deformation retracts.
\end{example}

\begin{example}
Dropping the second line of identities gives another example $\mcR$, this time with vertical maps homotopy retracts. A homological perturbation theory for homotopy retracts was worked out by \cite{Cranic}.
\end{example}

\begin{example}
There are perturbed versions $\Sps$ and $\Spt$ (and their counterparts for $\mcR$) of the above examples that include a further generator $\delta$ at the source or target subject to the Maurer--Cartan equation $D\delta = -\delta^2$. However, we do not see a way of making these into cocategory objects so, in the above picture, vertical composition is not defined.

We will also make use of the localization $\Spsloc$ at $1 + \delta h$, i.e.\ in addition to $\Sps$ a further generator $(1 + \delta h)^{-1}$ is added and forced inverse to $1 + \delta h$; then $(1 + h \delta)^{-1} = 1 - h (1 + \delta h)^{-1} \delta$ as one can easily check, so $1 + h \delta$ is also invertible.
\end{example}

\section{Perturbation lemma}

We will now reinterpret Proposition~\ref{prop:transfer_diagram} for the concrete example of strong deformation retractions, i.e.\ for $\mcA = \mcS$. For transferring the diagram as in the proposition, we need two perturbations, $\delta$ of the source and $\delta'$ of the target, such that $\delta' f = f \delta$, $\delta g = g \delta'$ and $\delta h = - h \delta$. All three conditions are satisfied if $\delta = g \delta' f$; it then follows that $\delta' = f \delta g$ is uniquely determined. For short, we will say that \emph{$\delta$ factors through $f$ and $g$}. On the other hand, if $\delta'$ is an arbitrary perturbation, then so is $\delta = g \delta' f$ and so this concept is equivalent to a perturbation of the target. We have thus obtained the following:

\begin{proposition} \label{prop:push_along}
Given the solid part of the diagram
\[\xymatrix{
A \ar[r]^-\theta \ar[d]_-{(f,g,h)} & A' \ar@{-->}[d] \\
B \ar@{-->}[r]_-\eta & B'
}\]
with $\theta$ a non-dg-isomorphism for which $D_\ell \theta$ factors through $f$ and $g$ there exists a universal filler square.
\end{proposition}

If desired, we may replace $A'$ and $B'$ by the isomorphic $A_{D_\ell \theta}$ and $B_{D_\ell \eta}$ so that this indeed yields a perturbation. As a particular special case, this applies to any dg-isomorphism $\theta$, in which case we may take the bottom map to be the identity.

Our approach to the perturbation lemma starts with the following situation
\[\xymatrix{
A \ar[r]^-\psi \ar[d]_-{(f,g,h)} & A_\delta \ar@{-->}[d] \\
B \ar@{-->}[r] & B'
}\]
where $(f,g,h)$ is given, as well as $\delta$, and we prescribe the latter via the non-dg-isomorphism $\psi$. Since generally $\delta = D_\ell \psi$ does not factor through $f$ and $g$, a filler does not exist. However, we may replace $\psi$ by any other non-dg-isomorphism $\theta$ and then a filler may exist. Such a non-dg-isomorphism $\theta$ may be interpreted as an isomorphism $A_\delta \cong A_{\hat\delta}$ and this, in effect, replaces $\delta$ by $\hat\delta = D_\ell \theta$. We will now find a way of constructing $\theta$ so that $\hat\delta$ factors through $f$ and $g$. For convenience, we will do so in two steps, first dealing only with factorization through $f$, but preserving any existing factorization through $g$, the second part is then formally dual and it remains to put the two parts together.

\begin{theorem}[Perturbation lemma]
Given a strong deformation retraction $A \to B$ and a perturbation $\delta$ of $A$, for which $1 + \delta h$ is invertible,
\[\xymatrix{
A \ar@<.5ex>[r]^-\psi \ar@<-.5ex>@{-->}[r]_-{\theta} \ar[d]_-{(f,g,h)} & A_\delta \ar@{-->}[d]^-{(\widehat f,\widehat g,\widehat h)} \\
B \ar@{-->}[r] & B_{\delta'}
}\]
there exists a non-dg-isomorphism $\theta \colon A \to A_\delta$ such that $D_\ell \theta$ factors through $f$ and $g$; consequently, a filler as above exists. There are explicit formulas
\[\widehat f = f (1 + \delta h)^{-1}, \quad \widehat g = (1 + h \delta)^{-1} g, \quad \widehat h = h (1 + \delta h)^{-1} = (1 + h \delta)^{-1} h\]
and $\delta' = \widehat f \delta g = f \delta \widehat g$. (To be perfectively precise, all these are maps between $A$ and $B$ and these are to be transfered using $\psi$ to maps between $A_\delta$ and $B_{\delta'}$.)
\end{theorem}

\begin{proof}
As explained, our first task is to find $\alpha_\delta \colon A \to A_\delta$, associated with a non-dg-automorphism $\alpha \colon A \to A$, satisfying two properties
\begin{itemize}
\item
	$\alpha g = g$ and
\item
	$D_\ell \alpha_\delta$ factors through $f$.
\end{itemize}
Using $D_\ell \alpha_\delta = D_\ell (\psi \alpha) = D_\ell \alpha + \alpha^{-1} \delta \alpha = \alpha^{-1} \cdot (D + \delta_*) \alpha$, the first condition and $D(1-gf)=0$, we can rewrite the second condition as
\begin{align*}
0 & = \alpha (D_\ell \alpha_\delta) (1-gf) = ((D+\delta_*)\alpha)(1-gf) \\
& = (D+\delta_*)(\alpha(1-gf)) =(D+\delta_*)(\alpha-gf).
\end{align*}
This shows that, assuming the first condition, the second condition is equivalent to $\alpha-gf$ being a cycle in $\mcC(A, A)_{\delta_*}$; a universal solution then has to be a cycle in $\mcS_\mathsf{ps}(s,s)_{\delta_*}$. Therefore, it is sufficient that $\alpha-gf$ be a boundary annihilated by $g$ from the right,
\[\alpha-gf=(D+\delta_*)\zeta\]
with $\zeta$ any non-dg-map of degree $1$ satisfying $\zeta g = 0$.\footnote{In fact, $\mcS_\mathsf{ps}(s,s)_{\delta_*}$ is acyclic, as well as its subcomplex formed by the left annihilators of $g$, so this condition is also necessary.} The simplest non-trivial case $\zeta=h$ yields
\[\alpha = gf + (D + \delta_*)h = gf + (1-gf) + \delta h = 1 + \delta h\]
(the absolutely simplest $\zeta =0$ would give $\alpha = gf$, which is only invertible for the trivial strong deformation retraction, so this is not much of a use). Returning to the equivalent form of the second condition, it implies easily that $(D+\delta_*)\alpha = \delta gf$ and finally $D_\ell \alpha_\delta = \alpha^{-1}\delta gf$. Therefore, if $\delta$ factors through $g$ then $D_\ell \alpha_\delta$ also factors through $g$ (since $\alpha^{-1} g = g$).

Imposing further that $\alpha h = h$, the above solution is unique: Since $\mcS_\mathsf{ps}(s,s)_1$ is generated by the compositions $\zeta = h \delta \dots \delta h$, we get easily that $\alpha = g f + (D + \delta_*) \zeta$ satisfies $\alpha h = \zeta$, hence we must have $\zeta = h$.

Dually, we find $\beta$ satisfying $f \beta = f$ and that $-D_r \beta^\delta$ factors through $g$. The simplest non-trivial solution (unique if we impose $h \beta = h$) is
\[\beta = gf + (D - \delta^*)h = gf + (1 - gf) + h \delta = 1 + h \delta\]
with $-D_r \beta^\delta = g f \delta \beta^{-1}$. If $\delta$ factors through $f$ then so does $-D_r \beta^\delta$.

It remains to combine the two steps. First we utilize the map $\alpha$ as on the left and further replace $\alpha_\delta$ by the composition of $\psi$ and $\alpha_\delta^{\delta_1}$ as on the right:
\[\xymatrix{
A \ar@<.5ex>[r]^-\psi \ar@<-.5ex>[r]_-{\alpha_\delta} \ar[d]_-{(f,g,h)} & A_\delta \ar@{-->}[d] & & A \ar[r]^-{\psi} \ar[d]_-{(f,g,h)} & A_{\delta_1} \ar[r]^-{\alpha^{\delta_1}_\delta}_-\cong \ar@{-->}[d] & A_{\delta} \ar@{-->}[d] \\
B \ar@{-->}[r] & B' & & B \ar@{-->}[r] & B' \ar[r]_-1^-\cong & B'
}\]
We obtain $\delta_1 = D_\ell \alpha_\delta = \alpha^{-1} \delta g f$ which factors through $f$. We may now reverse $\psi \colon A \to A_{\delta_1}$ and apply to it the dual argument, i.e.\ we utilize the map $\beta$ (we will see below that it is indeed invertible):
\[\xymatrix{
A \ar[d]_-{(f,g,h)} & A_{\delta_1} \ar@<-.5ex>[l]_-{\varphi} \ar@<.5ex>[l]^-{\beta^{\delta_1}} \ar[r]^-{\alpha^{\delta_1}_\delta}_-\cong \ar@{-->}[d] & A_{\delta} \ar@{-->}[d] & & A \ar[d]_-{(f,g,h)} & A_{\delta_2} \ar[l]_-{\varphi} \ar@{-->}[d] & A_{\delta_1} \ar[l]_-{\beta^{\delta_1}_{\delta_2}}^-\cong \ar[r]^-{\alpha^{\delta_1}_\delta}_-\cong \ar@{-->}[d] & A_{\delta} \ar@{-->}[d] \\
B & B' \ar@{-->}[l] \ar[r]_-1^-\cong & B' & & B & B_{\delta'} \ar@{-->}[l]^-\varphi & B_{\delta'} \ar[l]^-1_-\cong \ar[r]_-1^-\cong & B_{\delta'}
}\]
We obtain $\delta_2 = -D_r \beta^{\delta_1} = g f \delta_1 \beta^{-1} = g f \alpha^{-1} \delta g f$ (since $f \beta^{-1} = f$), which factors through $f$and $g$ and we may thus apply the above proposition. In addition, the formula for $\delta_2$ yields the formula $\delta' = f \alpha^{-1} \delta g$.

It remains to compute the new strong deformation retraction $(\widehat f, \widehat g, \widehat h)$. Pretending that $\varphi$ is the identity, all components are obtained from $(f,g,h)$ by composing with $\beta \alpha^{-1}$ on the right and/or $\alpha \beta^{-1}$ on the left. Thus, we are left to show the following:
\begin{align*}
\widehat f & = f \beta \alpha^{-1} = f \alpha^{-1} \\
\widehat g & = \alpha \beta^{-1} g = \beta^{-1} g \\
\widehat h & = \alpha \beta^{-1} h \beta \alpha^{-1} = h \alpha^{-1}
\end{align*}
Given that $\beta = 1 + h \alpha^{-1} \delta g f$, it is clear that $\beta^{-1} = 1 - h \alpha^{-1} \delta g f$ so that $\beta$ is indeed invertible, and also that $\beta h = h = h \beta$; we also have $\alpha h = h$. Finally, $\beta^{-1} g = (1 - h \alpha^{-1} \delta) g = (1 + h \delta)^{-1} g$ and this is annihilated by $h$ from the left so that $\alpha \beta^{-1} g = (1 + \delta h) (1 + h \delta)^{-1} g = (1 + h \delta)^{-1} g$.
\end{proof}


\begin{remark}
The resulting map $\theta \colon A \to A_\delta$ can be written more conveniently as
\[A \xlla{\beta} A \xlra{\alpha} A \xlra{\psi} A_\delta\]
though with individual ingredients not satisfying the assumptions of Proposition~\ref{prop:push_along}. This means that the new strong deformation retraction $(\widehat f, \widehat g, \widehat h)$ is obtained by pushing along $\alpha \beta^{-1}$ and then perturbing the differential, as in the fromulas from the perturbation lemma. We could have organized the proof similarly, by using the dual decomposition in the two steps, thus obtaining a potentially different map $\theta$, given by
\[A \xlra{\alpha} A \xlla{\beta_1} A \xlra{\psi} A_\delta\]
with $\beta_1$ obtained almost as above but from $\varphi \colon A_\delta \to A_{\delta'_1}$ (this would be a mild inconvenience since we would need to use a strong deformation retraction different from the original $(f,g,h)$), but it turns out that the resulting map $\theta$ would be the same.

On the other hand, we could run the whole argument dually, starting with the step $\beta$, followed by the step $\alpha$. Then the new strong deformation retraction would be obtained by pulling along the composition
\[A_\delta \xlra{\varphi} A \xlra{\overline\beta} A \xlla{\overline\alpha} A\]
(with slightly different maps $\alpha$ and $\beta$ since they are applied in the opposite order). It seems that this map is not inverse to $\theta$ above, thus yielding a different way of obtaining the perturbation lemma. However, the resulting strong deformation retraction is the same, as one can observe by looking at the formulas for $(\widehat f, \widehat g, \widehat h)$ and $\delta'$ that are self-dual. Thus, we do not expect the perturbation lemma to have any strong universal property, even though the two ingredients $\alpha$ and $\beta$ are unique.

In more detail, the composition $\alpha \beta^{-1}$ can be computed as
\[(1 + \delta h)(1 + \widehat h \delta g f)^{-1} = (1 + \delta h)(1 - \widehat h \delta g f) = 1 + \delta h - \widehat h \delta g f.\]
The dual composition $\overline\alpha^{-1} \overline\beta$ can be computed as
\[(1 + g f \delta \widehat h)^{-1}(1 + h \delta) = (1 - g f \delta \widehat h)(1 + h \delta) = 1 + h \delta - g f \delta \widehat h.\]
The composition of these two maps in the simpler direction $\alpha \beta^{-1} \overline\alpha^{-1} \overline\beta$ is
\[1 + \delta h + h \delta - g f + \widehat g \widehat f = [d + \delta, h] + \widehat g \widehat f\]
and it seems very unlikely that this equals $1$ (this would mean that $h$ is a homotopy between $\widehat g \widehat f$ and $1$ on $A_\delta$; given that $\widehat h$ is such a homotopy, we find it unlikely that $h$ would be one, too).
\end{remark}

\section{Functoriality of the perturbation lemma}

In its simplest form, we may view the perturbation lemma as an association
\[[\Spsloc, \mcC] \to [\mcS, \mcC]\]
for any dg-category $\mcC$ closed under perturbations.

\begin{theorem}
The above association is a dg-functor, which preserves all limits and colimits.
\end{theorem}

\begin{proof}
Let $Q$ denote the operator of the free closure under perturbations and consider the dg-functor
\[P \colon \mcS \otimes \bbZ[\delta] \to Q\Spsloc\]
given on objects by $s \mapsto s$ and $t \mapsto t$ and on generating morphisms $f \otimes 1_*$, $g \otimes 1_*$, $h \otimes 1_*$ coming from $\mcS$ by the formulas for $(\widehat f, \widehat g, \widehat h)$; we know that these form a strong deformation retraction.\footnote{One may understand the proof of the perturbation lemma in $Q\Spsloc$, in which case it produces exactly the required dg-functor.} The remaining generators $1_s \otimes \delta$ and $1_t \otimes \delta$ are sent to $\delta$ and $\delta'$. This dg-functor then induces
\[[\Spsloc, \mcC] \cong [Q\Spsloc, \mcC] \xlra{P^*} [\mcS \otimes \bbZ[\delta], \mcC] \cong [\mcS, [\bbZ[\delta], \mcC]] \xlra{\{W, -\}_{\bbZ[\delta]*}} [\mcS, \mcC]\]
and this is clearly a dg-functor. Since $P^*$ and the weighted limit both preserve limits, the same is true for the composition. The colimit statement comes from the same argument but with the perturbation interpreted as a weighted colimit.
\end{proof}

In our double-categorical picture, the above theorem only concerns the horizontal composition, since the source is not equipped with a vertical composition. The compatibility with vertical is thus more subtle and will be our next concern. After that, we will deal with iterating the perturbation lemma and with its compatibility with the tensor structure.

In order to express the compatibility with the vertical composition, we interpret the perturbation lemma as an association:
\[\xymatrix{
A \ar[r]^-\psi \ar[d]_-{(f,g,h)} & A_\delta & & {} \ar@{}[d]|-{\displaystyle\longmapsto} & & A \ar[r]^-{\theta} \ar[d]_-{(f,g,h)} & A_\delta \ar[d]^-{(f_1,g_1,h_1)} \\
B & & & & & B \ar[r]_-\psi & B_{\delta'}
}\]
Now consider the following situation
\[\xymatrix{
A \ar[r]^-\psi \ar[d]_-{(f,g,h)} & A_\delta & A \ar[r]^-{\theta} \ar[d] & A_\delta \ar[d] & A \ar[r]^-{\theta} \ar[d] & A_\delta \ar[dd] \\
C \ar[d]_-{(f',g',h')} & & C \ar[d] \ar@<.5ex>[r]^-{\psi} \ar@<-.5ex>[r]_-{\eta} & C_{\delta'} \ar[d] & C \ar[d] & \\
B & & B \ar[r]_-{\psi} & B_{\delta''} & B \ar[r]_-{\psi} & B_{\hat\delta}
}\]
with two composable strong deformation retractions and a perturbation of the source. We may either apply the perturbation lemma to the top square, thus obtain $\psi \colon C \to C_{\delta'}$ and then apply the perturbation lemma to it; or we can apply the perturbation lemma directly to the composite strong deformation retraction. Note that the two maps $\theta$ at the top are different -- they are induced by the same perturbation of $A$ but with respect to different strong deformation retractions: $A \to C$ in the first case and the composite $A \to B$ in the second case. One can thus only hope that the resulting strong deformation retractions $A_\delta \to B_{\delta''}$ are equal (one cannot even compose the squares in the first case).

\begin{theorem}
The two strong deformation retractions $A_\delta \to B_{\delta''}$ are equal.
\end{theorem}

\begin{proof}
We denote the perturbed sdr's as $(f_1, g_1, h_1)$ and $(f_1', g_1', h_1')$ and the composite one as $((f'f)_1, (gg')_1, (h+gh'f)_1)$. In order to compare the resulting components $f$ and $h$, it is enough to push along the map $\alpha$ (and thus ignore $\beta$). For the composite sdr, this means pushing along $1+\delta(h+gh'f)$ on the source and $1$ on the target which we can accompany by pushing along $1+\delta' h'$ on the middle object yielding
\[\xymatrix@C=4pc{
A \ar[r]^-{1+\delta(h+gh'f)} \ar[d]_-{f} & A_\delta \ar[d]^-{f_1} \\
C \ar[r]^-{1+\delta'h'} \ar[d]_-{f'} & C_{\delta'} \ar[d]^-{f_1'} \\
B \ar[r]^-1 & B_{\delta''}
}\]
Recalling that $\delta' = f \alpha^{-1} \delta g = f_1 \delta g$ the indicated $f$-components clearly work out (the lower square is the definition of $f_1'$, while the upper square differs from the defining one by $\delta g h' f$ on the top and $f_1 \delta g h'$ on the bottom and these clearly form a commutative square). This easily implies that $(f' f)_1 = f_1' f_1$ and the dual argument resolves the $g$-components. It remains to deal with the homotopies. Pushing instead along $1$ on the middle object yields
\[\xymatrix@C=4pc{
A \ar[r]^-{1+\delta(h+gh'f)} \ar[d]_-{(f,g,h)} & A_\delta \ar[d]^-{(f_2,g_2,h_2)} \\
C \ar[r]^-{1} \ar[d]_-{(f',g',h')} & C \ar[d]^-{(f',g',h')} \\
B \ar[r]^-1 & B
}\]
(note that these do not satisfy the dg-conditions since the necessary condition of $D_\ell\alpha$ commuting with maps in the image is not satisfied). The composition on the right has homotopy $h_2+g_2h'f_2$ and we wish to show that this equals $h_1+g_1h_1'f_1$. Now consider the pushouts of $(f,g,h)$ and $h'$ on the left side of the diagram below along the following pairs (the map $1 - \delta g h' f_2$ is computed as the composite of the remaining two maps in the commutative triangle at the top):
\[\xymatrix{
\bullet \ar[r]_-{1+\delta h} \ar@/^1pc/[rr]^-{1+\delta(h+gh'f)} \ar[d] & \bullet \ar[d] & \bullet \ar[l]^-{1 - \delta g h' f_2} \ar[r]_-1 \ar[d] & \bullet \ar[d] \\
\bullet \ar[r]_-1 & \bullet & \bullet \ar[l]^-1 \ar[r]_-{1 + f_1 \delta g h'} & \bullet \\
(f,g,h),h' & (f_1,g,h_1),h' & (f_2,g_2,h_2),h' & (f_1,?,?),h_1' \\
& & h'f_2 \ar@{|->}[r]^-{(1 + f_1 \delta g h')_*} & h_1'f_1 \\
& gh' & g_2h' \ar@{|->}[l]^-{(1 - \delta g h' f_2)_*}
}\]
Because of this, the indicated composites are mapped one to the other, but they are in fact equal, since $h'h'=0$ for the upper case and $h'f_2g_2h'=h'h'=0$ for the lower case. This enables us to replace $h_2 + g_2 h' f_2 = h_2 + g h' f_2$ and $h_1 + g_1 h_1' f_1 = h_1 + g_1 h' f_2$. Finally, to compare the right hand sides, we use the middle part of the above diagram to get a relation between $h_2$ and $h_1$:
\[h_2 = (1 - \underbrace{\delta g h' f_2) h_2}_0 = h_1 (1 - \delta g h' f_2) = h_1 - \underbrace{h_1 \delta g}_{g-g_1} h' f_2 = h_1 - (g - g_1) h' f_2,\]
which is just what we wanted.

Finally, we need to prove the equality $\delta'' = \hat\delta$ of the perturbations on targets. As observed in the proof of the perturbation lemma, the perturbation on the target equals $\widehat f \delta g$. Since we proved the two versions of $\widehat f$ to be equal, this yields equality of the perturbations.
\end{proof}

Now we will deal with an iterated application of the perturbation lemma; in the formalization above, this is functoriality in the horizontal direction. Assuming that both $\delta, \varepsilon \in \MC(A)$, we consider the following.
\[\xymatrix{
A \ar@<.5ex>[r]^-{\psi} \ar@<-.5ex>[r]_-{\theta} \ar[d] & A_\delta \ar@<.5ex>[r]^-{\psi} \ar@<-.5ex>[r]_-{\theta_\delta} \ar[d] & A_\varepsilon \ar[d] & & A \ar@<.5ex>[r]^-{\psi} \ar@<-.5ex>[r]_-{\theta} \ar[d] & A_\varepsilon \ar[d] \\
B \ar[r]_-\psi & B_{\delta'} \ar[r]_-\psi & B_{\varepsilon'} & & B \ar[r]_-\psi & B_{\varepsilon'}
}\]

\begin{theorem}
The two strong deformation retractions $A_\varepsilon \to B_{\varepsilon'}$ are equal.
\end{theorem}

\begin{proof}
As in the previous proof, we study pushouts along the $\alpha$-maps. In the second case, this is $1 + \varepsilon h$. In the first case, this is the composition
\[(1 + (\varepsilon - \delta) h_1) (1 + \delta h) = 1 + \delta h + (\varepsilon - \delta) h = 1 + \varepsilon h,\]
where $h_1 = h (1 + \delta h)^{-1}$ is the homotopy of the middle sdr, obtained by pushing $h$ along $1 + \delta h$.
\end{proof}

The tensor product of sdr's
\[F^\ell = (f^\ell,g^\ell,h^\ell) \colon A^\ell \to B^\ell, \quad F^r = (f^r, g^r, h^r) \colon A^r \to B^r\]
is given by
\[F^\ell \otimes F^r = (f^\ell \otimes f^r, g^\ell \otimes g^r, h^\ell \otimes 1 + g^\ell f^\ell \otimes h^r) \colon A^\ell \otimes A^r \to B^\ell \otimes B^r.\]
Note that this tensor product is not symmetric, but it can be shown to be associative. Alternatively, it can be described as the composition 
\[A^\ell \otimes A^r \xlra{F^\ell \otimes 1} B^\ell \otimes A^r \xlra{1 \otimes F^r} B^\ell \otimes B^r.\]
Given two perturbations $\delta^\ell$ on $A^\ell$ and $\delta^r$ on $A^r$, we obtain a perturbation $\delta^\ell \otimes 1 + 1 \otimes \delta^r$ on $A^\ell \otimes A^r$. For the following theorem, we define the result of the perturbation lemma applied to a sdr $F$ and a perturbation $\delta$ by $F_\delta$.

\begin{theorem}
$(F^\ell \otimes F^r)_{\delta^\ell \otimes 1 + 1 \otimes \delta^r} = F^\ell_{\delta^\ell} \otimes F^r_{\delta^r}$.
\end{theorem}

\begin{proof}
This follows from the previous two theorems and two special cases where a tensor multiple of an sdr $1_C \otimes F \colon C \otimes A \to C \otimes B$ is perturbed either by $1 \otimes \delta$ or $\delta \otimes 1$. In the first case, perturbing $1 \otimes F$ by $1 \otimes \delta$ yields $1_C \otimes F_\delta$ by cocontinuity of the perturbation lemma. In the second case, when perturbing $1 \otimes F$ by $\delta \otimes 1$, the components $f$ and $h$ are obtained by pushing along $\alpha = 1 \otimes 1 + (\delta \otimes 1)(1 \otimes h) = 1 \otimes 1 + \delta \otimes h$ and so $(1 \otimes f) \alpha ^{-1} = 1 \otimes f$ and similarly for $1 \otimes h$; this gives $(1_C \otimes F)_{\delta \otimes 1} = 1_{C_\delta} \otimes F$.
\end{proof}

\end{document}